\documentclass[11pt]{article}
\usepackage{mathrsfs}
\usepackage{amsfonts}
\usepackage{}
\usepackage{amsmath,amssymb,amsthm,latexsym,amstext}
\usepackage[mathscr]{eucal}
\usepackage{rotate,graphics,epsfig,epstopdf}
\usepackage{float}
\usepackage{color}
\usepackage{subfigure}
\usepackage{indentfirst}
\usepackage{bm}

\newcommand{\be}{\begin{eqnarray}}
\newcommand{\ee}{\end{eqnarray}}
\newcommand{\by}{\begin{eqnarray*}}
\newcommand{\ey}{\end{eqnarray*}}
\newcommand{\bn}{\begin{enumerate}}
\newcommand{\en}{\end{enumerate}}
\newcommand{\ei}{\end{itemize}}

\newtheorem{theorem}{Theorem}
\newtheorem{lemma}[theorem]{Lemma}

\newtheorem{remark}[theorem]{Remark}
\newtheorem{definition}[theorem]{Definition}
\newtheorem{proposition}[theorem]{Proposition}

\renewcommand{\theequation}{\arabic{section}.\arabic{equation}}

\numberwithin{equation}{section}

\begin{document}
\date{}
\title{\bf Second-order McKean-Vlasov stochastic evolution equation driven by Poisson jumps: existence, uniqueness and averaging principle  \footnote{This work was supported by
the Natural Science Foundation of Jiangsu Province, BK20230899 and
the National Natural Science Foundation of China, 11771207.
}}
\author{ Chungang Shi \footnote{shichungang@njust.edu.cn }\\
\texttt{{\scriptsize School of Mathematics and Statistics, Nanjing University of Science and Technology,
Nanjing, 210094, P. R. China}}}\maketitle
\begin{abstract}
In the paper, a class of second-order McKean-Vlasov stochastic evolution equation driven by Poisson jumps with non-Lipschitz conditions is considered. The existence and uniqueness of the mild solution is established by means of the Carath${\rm \acute{e}}$odory approximation technique.  Furthermore, an averaging principle is obtained between the solution of the second-order McKean-Vlasov stochastic evolution equation and that of the simplified equation in mean square sense.
\end{abstract}



\section{Introduction}
  \setcounter{equation}{0}
  \renewcommand{\theequation}
{1.\arabic{equation}}
Stochastic functional partial differential equation is a natural choice in mathematical modeling of some phenomena in natural sciences~\cite{DZ,So,Ch}. Further, distribution-dependent first-order stochastic (partial) differential equations in finite and infinite dimensions appear naturally in the study of diffusion processes~\cite{SXW, AD, MM1, Ah}. Subsequently, this class of equations are extended to second-order case~\cite{MW1}. The study of abstract deterministic second-order evolution equations by mean of the cosine family is initialed by Travis and Webb~\cite{TW1}. Whereafter, Travis and Webb~\cite{TW2} and Fattorini~\cite{F} have systematically studied deterministic second-order linear and nonlinear equations with the help of cosine operator theory. And then Morales~\cite{Mo} studied the existence of mild and classic solutions for an abstract second-order impulsive Cauchy problem by the same method.  

For the stochastic case, it is well known that Brownian motion is a continuous stochastic process which is not suitable for describing some real systems whose structures are subject to stochastic abrupt changes. Such as stochastic failures, repairs of the components, changes in the interconnections and sudden environmental changes which may originate from abrupt phenomena. For these discontinuous systems, stochastic differential equation driven by L${\rm\acute{e}}$vy noise is a right model. Furthermore, the second-order stochastic differential equations are often used to account for integrated processes that can be made stationary in continuous time. For example, it is useful for engineers to model mechanical vibrations or charge on a capacitor or condenser subjected to noise excitation by a second-order stochastic differential equations. There are numerous works on the second-order stochastic differential equation. Mahmudov and McKibben~\cite{MM2} considered a class of abstract second-order damped McKean-Vlasov stochastic evolution equations and the global existence, uniqueness of mild solutions was derived under the global Lipschitz and linear growth conditions for coefficients. Ren and Sun~\cite{RS2} investigated a class of second-order neutral stochastic evolution equations with infinite delay under Carath${\rm\acute{e}}$odory conditions by successive approximation. Then Ren and Sakthivel~\cite{RS1} studied the existence, uniqueness and stability of the mild solutions for second-order neutral stochastic evolution equations with infinite delay and Poisson jumps under non-Lipschitz conditions. Yue~\cite{Y} considered a class of second-order neutral impulsive stochastic evolution equation with infinite delay and the existence of mild solutions was obtained by means of the Krasnoselskii-Schaefer fixed point theorem under Lipschitz conditions. McKibben and Webster~\cite{MW2} studied a class of second-order McKean-Vlasov stochastic evolutions driven by fractional Brownian motion and Poisson jumps and the global existence and uniqueness of the mild solutions by utilizing the successive approximation method was obtained under various growth conditions.

On the other hand, the averaging method is a powerful tool for reducing system complexity which allows the original complex time-varying system to be replaced by a simplified autonomous averaging system. The original work of the averaging principle is attributed to Khasminskii~\cite{Kh}. There are numerous works in the research area, such as~\cite{SWY, KA, AL}. Specially, Wang et al.~\cite{WYHG} derives an averaging principle for McKean-Vlasov-type Caputo fractional stochastic differential equations. Abouagwa et al.~\cite{AABK} studied a class of mixed neutral Caputo fractional stochastic evolution equations with infinite delay and the existence, uniqueness and averaging principle was established. Guo et al.~\cite{GXWH} considered the averaging principle for a class of stochastic differential equations with the nonlinear terms only satisfying the local Lipschitz and monotone conditions. Chao et al.~\cite{CDGW} investigated a class of McKean-Vlasov stochastic differential equations driven by L${\rm\acute{e}}$vy-type perturbations and the existence and uniqueness of the solutions were established by utilizing the Euler-like approximation. Moreover, the averaging principle was derived for the addressed system in the sense of mean square convergence. Shen et al.~\cite{SXW2} considered a multi-valued McKean-Vlasov stochastic differential equations and obtained a stochastic averaging principle in the same sense.

Motivated by the above researches, in the paper, we consider the following second-order McKean-Vlasov stochastic evolution equation with Poisson jumps
\begin{eqnarray}\label{main}
dX'(t)+[BX'(t)+AX(t)]dt&=&F(t,X(t),\mathcal{L}(X(t)))dt\nonumber\\
&+&G(t,X(t),\mathcal{L}(X(t)))dW(t)\nonumber\\
&+&\int_{Z}J(t-,X(t-),\mathcal{L}(X(t-)),z)\tilde{N}(dt,dz),\nonumber\label{main}\\
X(0)&=&X_{0},\quad X'(0)=X_{1},\nonumber\quad 0\leq t\leq T,\\
\mathcal{L}(X(t))&=&\text{probability distribution of}~X(t).
\end{eqnarray}
Here $H,K$ are real separable Hilbert space, $W(t)$ is the $K$-valued Wiener process with a positive, nuclear covariance operator $Q$ defined on a complete probability space $(\Omega, \mathcal{F}, \mathbb{P})$ equipped with a normal filtration $\{\mathcal{F}_{t}\}_{t\geq0}$ generated by $W$ and $N$. The linear operator $A: D(A)\subset H\to H$ generates a strongly continuous cosine family on $H$. $B: H\to H$ is a bounded linear operator, $F: [0,T]\times H\times \mathcal{P}_{2}(H)\to H$ and $G:[0,T]\times H\times\mathcal{P}_{2}(H)\to L(K,H)$. $X_{0}$ is a $\mathcal{F}_{0}$-measurable $E$-valued~(defined later) random variable and $X_{1}$ is a $H$-valued random variable. Further, they are independent of $W$ and $N$ with finite second moments. 

We study the existence and uniqueness of the mild solutions for equation (\ref{main}) by means of Carath${\rm\acute{e}}$odory approximation technique which is used in~\cite{SXW}. It is worth noting that, in~\cite{SXW}, Shen et al. established the existence and uniqueness of the strong solution for a distribution dependent stochastic differential equation driven by mixed Gaussian noises. Here, we intend to apply the method to second-order McKean-Vlasov stochastic differential equation with discontinuous noise to prove the existence and uniqueness of mild solutions. For the continuous Gaussian noise, Mahmudov and McKibben~\cite{MM2} obtained the global existence and uniqueness of mild solutions under Lipschitz condition by the fixed point theorem. For the first-order abstract stochastic McKean-Vlasov evolution equations, Mahmudov and McKibben~\cite{MM1} established results concerning the global existence, uniqueness of mild solutions under Carath${\rm\acute{e}}$odory conditions which is different from our method. More generally, compared with the recursive approximations method and results in McKibben and Webster~\cite{MW2}, the advantage of the Carath${\rm\acute{e}}$odory approximation method is that we don't need to compute $X_{i}(t),~i=1,2,\cdots,k-1$ to compute $X_{k}(t)$ and our obtained results for mild solution are uniform on time. Lastly, for the established averaging principle between the solution of the second-order McKean-Vlasov stochastic evolution equation and that of the simplified equation in mean square sense, our averaging conditions for coefficients is more general than those in \cite{SXW, GXWH, CDGW}.

The paper is organized as follows. In section \ref{sec:preli}, we present some notations, assumptions and results for later use. In section \ref{EU}, The existence and uniqueness of the mild solutions is established. An averaging principle is derived in section \ref{Aver} between the original equation and the averaged equation.

\section{Preliminary}\label{sec:preli}
  \setcounter{equation}{0}
  \renewcommand{\theequation}
{2.\arabic{equation}}

Let $H,K$ denote the real separable Hilbert space with norm $\|\cdot\|$ and $\|\cdot\|_{K}$ respectively. $L(K,H)$ denotes the space consisting of the bounded linear operator from $K$ to $H$ and if $K=H$, then it is written as $L(H)$. If there is no confusion, let's write the norm in $H, K$ and $L(K, H)$ as $\|\cdot\|$.  Let $(\Omega,\mathcal{F},\mathbb{P})$ denote the complete probability space. $\{\mathcal{F}_{t}\}_{t\geq0}$ is the normal filtration which is generated by $K$-valued Wiener process $W$ having a positive, nuclear covariance operator $Q$, and Poisson process $N$. And let $\mathcal{F}_{T}=\mathcal{F}$ and $K_{0}=Q^{\frac{1}{2}}(K)$. We also introduce the space $L_{2}^{0}=L_{2}(K_{0},H)$ consisting of all Hilbert-Schmidt operator from $K_{0}$ to $H$ equipped with the norm $\|\cdot\|_{L_{2}^{0}}$. Let $\{p(t)\}_{t\geq0}$ be $\sigma$-finite, stationary, $\mathcal{F}_{t}$-adapted Poisson point process and take value in measurable space $(Z,\mathcal{B}(Z))$. Define a Poisson random measure which is induced by $\{p(t)\}$
\begin{eqnarray*}
N((s_{1},s_{2}], U)=\Sigma_{s\in(s_{1},s_{2}]}I_{U}(p(s)).
\end{eqnarray*}
For any $U\in \mathcal{B}(Z\backslash\{0\})$, define the compensated Poisson random measure as
\begin{eqnarray*}
\tilde{N}(dt,dx)=N(dt,dx)-\nu(dx)dt,   
\end{eqnarray*}
where $\nu$ is the intensity measure of $N$. 
Let $\mathcal{P}_{2}(H)$ denote the space of Borel probability measures on $H$ with 
\begin{equation*}
\int_{H}\|x\|^{2}\mu(dx)<\infty
\end{equation*} 
for every $\mu\in\mathcal{P}_{2}$. For any $\mu, \nu$ in $\mathcal{P}_{2}(H)$, define the following Monge-Kantorovich (or Wasserstein) distance:
\begin{equation}\label{dMK}
\mathcal{W}_{2}(\mu,\nu)=\inf_{\pi\in \Pi(\mu,\nu)}\Big[\int\int_{H\times H}\|x-y\|^{2}\Pi(dx,dy)\Big]^{\frac{1}{2}},
\end{equation}
where $\Pi(\mu,\nu)$ is the set of Borel probability measures $\pi$ on $H\times H$ with first and second marginals $\mu$ and $\nu$. Equivalently, for $\mu,\nu\in\mathcal{P}_{2}$,
\begin{equation}\label{rvMK}
\mathcal{W}_{2}(\mu,\nu)=\inf_{(X,\bar{X})}[\mathbb{E}\|X-\bar{X}\|^{2}]^{\frac{1}{2}}
\end{equation}
with random variables $X$ and $\bar{X}$ in $H$ having laws $\mu$ and $\nu$, respectively. For any fixed $T>0$, let $D([0,T]; H)$ denote the space consisting of all $H$-valued ${\rm c\grave{a}dl\grave{a}g}$ processes on $[0,T]$ which is a Banach space equipped with the supremum norm~\cite{EK}. Moreover,  let $L^{2}(\Omega; D([0,T]; H))$ be the totality of $D([0,T]; H)$-valued random variables $Y$ satisfying $\mathbb{E}\sup_{0\leq t\leq T}\|Y(t)\|^{2}<\infty$. Then $L^{2}(\Omega; D([0,T]; H))$ is a Banach space under the norm~\cite{GZ}
\begin{eqnarray*}
\|Y\|_{L^{2}}=(\mathbb{E}\sup_{0\leq t\leq T}\|Y(t)\|^{2})^{\frac{1}{2}}.
\end{eqnarray*}

Next we give some facts on cosine operator family~\cite{TW1,TW2,F}.
\begin{definition}
{\rm(i)} The one-parameter family $\{C(t):t\in \mathbb{R}\}\subseteq L(H)$ satisfying\\
{\rm(a)} C(0)=I,\\
{\rm(b)} C(t)x is continuous with respect to $t$ on $\mathbb{R}$ for any $x\in H$,\\
{\rm(c)} C(t+s)+C(t-s)=2C(t)C(s), for any $s,t\in \mathbb{R}$, is called strongly continuous cosine family.\\
{\rm(ii)} The corresponding strongly continuous sine family $\{S(t):t\in \mathbb{R}\}\subseteq L(H)$ is defined by $S(t)x=\int_{0}^{t}C(s)xds$ for any $t\in\mathbb{R}, x\in H$.
\end{definition}
\begin{definition}
The infinitesimal generator $A: H\to H$ is defined by 
\begin{eqnarray*}
Ax=\frac{d^{2}}{dt^{2}}C(t)x|_{t=0},
\end{eqnarray*}
for any $x\in D(A)=\{x\in H: C(\cdot)x\in C^{2}(\mathbb{R},H)\}$.
\end{definition}
We introduce a space $E=\{x\in H: C(t)x\in C^{1}(\mathbb{R},H)\}$, then obviously $D(A)\subset E$. It is well known that the infinitesimal generator $A$ is a closed, densely defined operator on $H$ and its cosine and sine family satisfies the following properties:
\begin{proposition}\label{prCS}
Assume that $A$ is the infinitesimal generator of a cosine family $\{C(t):t\in\mathbb{R}\}$. Then the following properties hold,\\
{\rm(1)} There exists $M_{A}\geq 1$ and $\omega\geq0$ such that $\|C(t)\|\leq M_{A}e^{\omega|t|}$ and hence $\|S(t)\|\leq M_{A}e^{\omega|t|}$.\\
{\rm(2)} $A\int_{s}^{r}S(u)x=[C(r)-C(s)]x$, for all $0\leq s\leq r<\infty$ and $x\in H$.\\
{\rm(3)} There exists $N\leq 1$ such that $\|S(s)-S(r)\|\leq N|\int_{s}^{r}e^{\omega|u|}du|$, for all $0\leq s\leq r<\infty$.\\
{\rm(4)} If $x\in E$, then $\lim_{t\to0}AS(t)x=0$.\label{proCS4}\\
{\rm(5)} $C(t+s)-C(t-s)=2AS(t)S(s)$, for any $s,t\in\mathbb{R}$.\label{CZS}
\end{proposition}
\begin{remark}
By the uniformly bounded principle and Proposition \ref{prCS}, we obtain that $\{C(t):t\in[0,T]\}$, $\{S(t):t\in[0,T]\}$ and $\{AS(t):t\in[0,T]\}$ are uniformly bounded,
hence there exists $N_{S}>0$ such that
\begin{eqnarray*}
\|S(t)-S(r)\|\leq N_{S}|t-r|,\quad 0\leq r\leq t\leq T.
\end{eqnarray*}
\end{remark}

\begin{lemma}[\cite{Ma}]\label{ME}
Let $p\geq2$, $r>0$ and $a, b\in\mathbb{R}$, then
\begin{eqnarray*}
|a+b|^{p}\leq [1+r^{\frac{1}{p-1}}]^{p-1}\Big(|a|^{p}+\frac{|b|^{p}}{r}\Big).
\end{eqnarray*}
\end{lemma}

\begin{lemma}[\cite{A}]{\rm(Kunita's first inequality)}\label{KN}
For any $p\geq 2$, there exists a constant $C_{p}>0$ such that
\begin{eqnarray*}
\mathbb{E}\sup_{0\leq s\leq t}\Big\|\int_{0}^{s}\int_{Z}J(\tau,z)\tilde{N}(d\tau,dz)\Big\|^{p}&\leq& C_{p}\Big\{\mathbb{E}\Big[\Big(\int_{0}^{s}\int_{Z}\|J(\tau,z)\|^{2}\nu(dz)d\tau\Big)^{\frac{p}{2}}\Big]\\
&&+\mathbb{E}\int_{0}^{s}\int_{Z}\|J(\tau,z)\|^{p}\nu(dz)d\tau\Big\}.
\end{eqnarray*}
\end{lemma}

\begin{lemma}[\cite{RS1}]{\rm(Bihari's inequality)}
Assume that $T>0, u_{0}\geq0$ and $u(t),v(t)$ are continuous function on $[0,T]$. Further, let $\theta:[0,\infty)\to[0,\infty)$ be continuous, nondecreasing concave function such that $\theta(r)>0$ for any $r>0$, if 
\begin{eqnarray*}
u(t)\leq u_{0}+\int_{0}^{t}v(s)\theta(u(s))ds,\quad 0\leq t\leq T,
\end{eqnarray*}
then 
\begin{eqnarray*}
u(t)\leq G^{-1}\Big(G(u_{0})+\int_{0}^{t}v(s)ds\Big), \quad 0\leq t\leq T,
\end{eqnarray*}
and $G(u_{0})+\int_{0}^{t}v(s)ds\in Dom(G^{-1})$, here $G(r)=\int_{1}^{r}\frac{ds}{\theta(s)},r\geq0$, $G^{-1}$ represents the inverse of $G$. Furthermore, if $u_{0}=0$, and $\int_{0}^{\infty}\frac{1}{\theta(s)}ds=\infty$, then $u(t)\equiv0, 0\leq t\leq T$. 
\end{lemma}

\section{Existence and Uniqueness}\label{EU}
  \setcounter{equation}{0}
  \renewcommand{\theequation}
{3.\arabic{equation}}

We give some assumptions with respect to coefficients which are used through the paper as follows,

$(\mathbf{H_{1}})$ $A$ is the infinitesimal generator of a strongly continuous cosine family $\{C(t):t\geq 0\}$ and the cosine family and the corresponding sine family $\{S(t):t\geq 0\}$ satisfy
\begin{eqnarray*}
\sup_{0\leq t\leq T}[\|S(t)\|^{2}+\|C(t)\|^{2}]\leq M_{1},
\end{eqnarray*}
where $M_{1}>0$.

$(\mathbf{H_{2}})$ $B: H\to H$ is a bounded linear operator with upper bound $M_{B}>0$.

$(\mathbf{H_{3}})$ $F:[0,T]\times H\times \mathcal{P}_{2}(H)\to H$, $G: [0,T]\times H\times \mathcal{P}_{2}(H)\to H$ and $J:[0,T]\times H\times\mathcal{P}_{2}(H)\times Z\to H$ are $\mathcal{F}_{t}$-adapted process and satisfy that there exists a nondecreasing bounded function $K(t)$ such that for any $t\in [0,T], x,y\in H$, $\mu,\nu\in \mathcal{P}_{2}(H)$, we have 
\begin{eqnarray*}
&&\|F(t,x,\mu)-F(t,y,\nu)\|^{2}+\|G(t,x,\mu)-G(t,y,\nu)\|_{L_{2}^{0}}^{2}\\
&&+\int_{Z}\|J(t,x,\mu,z)-J(t,y,\nu,z)\|^{2}\nu(dz)ds\leq K(t)\Psi(\|x-y\|^{2}+\mathcal{W}_{2}(\mu,\nu)^{2})
\end{eqnarray*} 
and
\begin{eqnarray*}
\|F(t,0,\delta_{0})\|^{2}+\|G(t,0,\delta_{0})\|_{L_{2}^{0}}^{2}+\int_{Z}\|H(t,0,\delta_{0},z)\|^{2}\nu(dz)\leq K(t),
\end{eqnarray*}
where $\Psi: \mathbb{R}^{+}\to\mathbb{R}^{+}$ is a continuous, nondecreasing, concave function and $\Psi(0)=0, \Psi(x)>0,$ for any $x>0$ such that $\int_{0^{+}}\frac{1}{\Psi(x)}dx=\infty$.

\begin{remark}
We give some concrete examples for function $\Psi$. Let $\gamma>0$ and $\delta\in(0,1)$, define
\begin{eqnarray*}
\Psi_{1}(u)=\gamma u,\quad u\geq0;
\end{eqnarray*}
\begin{equation*}
\Psi_{2}(u)=
\begin{cases}
u\log(u^{-1}),\quad 0\leq u\leq \delta,\\
\delta\log(\delta^{-1})+\Psi'_{2}(\delta-)(u-\delta),\quad u>\delta;
\end{cases}
\end{equation*}
\begin{equation*}
\Psi_{3}(u)=
\begin{cases}
u\log(u^{-1})\log\log(u^{-1}),\quad 0\leq u\leq \delta,\\
\delta\log(\delta^{-1})\log\log(\delta^{-1})+\Psi'_{3}(\delta-)(u-\delta),\quad u>\delta,
\end{cases}
\end{equation*}
where $\Psi'_{i}, i=2,3$ denote the derivative of $\Psi_{i}$ and they are nondecreasing concave function satisfying $\int_{0^{+}}\frac{1}{\Psi_{i}(u)}du=\infty$. Further, we note that the Lipschitz function is a special case.
\end{remark}
\begin{definition}
A ${\rm c\grave{a}dl\grave{a}g}$ stochastic process $X: [0,T]\times\Omega\to H$ is called a mild solution to equation (\ref{main}), if \\
{\rm(1)} $X(t)$ is $\mathcal{F}_{t}$-adapted, for any $0\leq t\leq T$,\\
{\rm(2)} $\int_{0}^{T}\|X(t)\|^{2}dt<\infty$, a.s.,\\
{\rm(3)} $X(t)$ satisfies
\begin{eqnarray*}
X(t)&=&S(t)X_{1}+[C(t)-S(t)B]X_{0}+\int_{0}^{t}C(t-s)BX(s)ds\\
&&+\int_{0}^{t}S(t-s)F(s,X(s),\mathcal{L}(X(s)))ds\\
&&+\int_{0}^{t}S(t-s)G(s,X(s),\mathcal{L}(X(s)))dW(s)\\
&&+\int_{0}^{t}\int_{Z}S(t-s)J(s-,X(s-),\mathcal{L}(X(s-)),z)\tilde{N}(ds,dz),\quad a.s.
\end{eqnarray*}
where $0\leq t\leq T$ and $\mathcal{L}(X(t))$ is the distribution of $X(t)$.
\end{definition}
Next we define the Carath${\rm \acute{e}}$odory approximation sequence as follows. For any integer $k\geq1$, define $X_{k}(t)=X_{0}$, for any $-1\leq t\leq 0$ and  
\begin{eqnarray}\label{Ca}
&&X_{0}(t)=S(t)X_{1}+(C(t)-S(t)B)X_{0},\quad 0\leq t\leq T,\nonumber\\
&&X_{k}(t)=S(t)X_{1}+(C(t)-S(t)B)X_{0}+\int_{0}^{t}C(t-s)BX_{k}(s-\frac{1}{k})ds\nonumber\\
&&+\int_{0}^{t}S(t-s)F(s,X_{k}(s-\frac{1}{k}),\mathcal{L}(X_{k}(s-\frac{1}{k})))ds\nonumber\\
&&+\int_{0}^{t}S(t-s)G(s,X_{k}(s-\frac{1}{k}),\mathcal{L}(X_{k}(s-\frac{1}{k})))dW(s)\\
&&+\int_{0}^{t}\int_{Z}S(t-s)J(s-,X_{k}((s-\frac{1}{k})-),\mathcal{L}(X_{k}(s-\frac{1}{k})-),z)\tilde{N}(ds,dz),\quad k\geq1, t\in[0,T]\nonumber.
\end{eqnarray}
Next we show that the sequence of stochastic processes $\{X_{k}(t)\}_{k\geq1}$ is uniformly bounded.
\begin{lemma}\label{SUB}
Suppose that $(\mathbf{H_{1}})$-$(\mathbf{H_{3}})$ are valid, then for any $k\geq 1$, there exists a constant $C_{1}>0$ such that 
\begin{eqnarray*}
\mathbb{E}\sup_{0\leq t\leq T}\|X_{k}(t)\|^{2}\leq C_{1}.
\end{eqnarray*}
\end{lemma}
\begin{proof}
Note that by (\ref{Ca}),
\begin{eqnarray}
&&\mathbb{E}\sup_{0\leq t\leq T}\|X_{k}(t)\|^{2}\leq 6\mathbb{E}\sup_{0\leq t\leq T}\|S(t)X_{1}\|^{2}+6\mathbb{E}\sup_{0\leq t\leq T}\|[C(t)-S(t)B]X_{0}\|^{2}\nonumber\\
&&\quad+6\mathbb{E}\sup_{0\leq t\leq T}\Big\|\int_{0}^{t}C(t-s)BX_{k}(s-\frac{1}{k})ds\Big\|^{2}\nonumber\\
&&\quad+6\mathbb{E}\sup_{0\leq t\leq T}\Big\|\int_{0}^{t}S(t-s)F(s,X_{k}(s-\frac{1}{k}),\mathcal{L}(X_{k}(s-\frac{1}{k})))ds\Big\|^{2}\nonumber
\end{eqnarray}
\begin{eqnarray}\label{UB}
&&\quad+6\mathbb{E}\sup_{0\leq t\leq T}\Big\|\int_{0}^{t}S(t-s)G(s,X_{k}(s-\frac{1}{k}),\mathcal{L}(X_{k}(s-\frac{1}{k})))dW(s)\Big\|^{2}\nonumber\\
&&\quad+6\mathbb{E}\sup_{0\leq t\leq T}\Big\|\int_{0}^{t}S(t-s)J(s-,X_{k}((s-\frac{1}{k})-),\mathcal{L}(X_{k}((s-\frac{1}{k})-)),z)\tilde{N}(ds,dz)\Big\|^{2}\nonumber\\
&&\quad\triangleq 6\sum_{i=1}^{6}I_{i}(t).
\end{eqnarray}
Next let's estimate each of these terms respectively. For $I_{i}(t), i=1, 2, 3$, $(\mathbf{H_{1}})$ and the H${\rm\ddot{o}}$lder inequality yield
\begin{eqnarray}\label{UB1}
I_{1}(t)&\leq& 6\mathbb{E}\sup_{0\leq t\leq T}\|S(t)X_{1}\|^{2}\leq 6M_{1}\mathbb{E}\|X_{1}\|^{2}\label{UB1}.\\
I_{2}(t)&\leq& 6\mathbb{E}\sup_{0\leq t\leq T}\|[C(t)-S(t)B]X_{0}\|^{2}\leq 12M_{1}(1+M_{B}^{2})\mathbb{E}\|X_{0}\|^{2}.\label{UB2}\\
I_{3}(t)&\leq&6\mathbb{E}\sup_{0\leq t\leq T}\Big\|\int_{0}^{t}C(t-s)BX_{k}(s-\frac{1}{k})ds\Big\|^{2}\nonumber\\
&\leq& 6TM_{B}^{2}M_{1}\int_{0}^{T}\mathbb{E}\|X_{k}(s-\frac{1}{k})\|^{2}ds.\label{UB3}
\end{eqnarray}
For $I_{4}(t)$, $(\mathbf{H_{1}})$-$(\mathbf{H_{3}})$ and the ${\rm H\ddot{o}lder}$ inequality yield
\begin{eqnarray}
I_{4}(t)&\leq&12\mathbb{E}\sup_{0\leq t\leq T}\Big(\int_{0}^{t}\|S(t-s)\|\|F(s,X_{k}(s-\frac{1}{k}),\mathcal{L}(X_{k}(s-\frac{1}{k})))-F(s,0,\delta_{0})\|ds\Big)^{2}\nonumber\\
&&+12\mathbb{E}\sup_{0\leq t\leq T}\Big(\int_{0}^{t}\|S(t-s)\|\|F(s,0,\delta_{0})\|ds\Big)^{2}\nonumber\\
&\leq&C_{T}M_{1}\mathbb{E}\int_{0}^{T}K(s)[\Psi(\|X_{k}(s-\frac{1}{k})\|^{2}+\mathcal{W}_{2}(\mathcal{L}(X_{k}(s-\frac{1}{k})),\delta_{0})^{2})+1]ds\nonumber\\
&\leq&C_{T}M_{1}\int_{0}^{T}K(s)(\Psi(2\mathbb{E}\|X_{k}(s-\frac{1}{k})\|^{2})+1)ds.\label{UB4}
\end{eqnarray}
For $I_{5}(t)$, the ${\rm BurkH\ddot{o}lder}$-Davis-Gundy inequality yields
\begin{eqnarray}
I_{5}(t)&\leq&6\mathbb{E}\sup_{0\leq t\leq T}\Big\|\int_{0}^{t}S(t-s)G(s,X_{k}(s-\frac{1}{k}),\mathcal{L}(X_{k}(s-\frac{1}{k})))dW(s)\Big\|^{2}\nonumber\\
&\leq&C_{T}\int_{0}^{T}\mathbb{E}\|S(t-s)G(s,X_{k}(s-\frac{1}{k}),\mathcal{L}(X_{k}(s-\frac{1}{k})))\|_{L_{2}^{0}}^{2}ds\nonumber\\
&\leq&C_{T}M_{1}\int_{0}^{T}\mathbb{E}\|G(s,X_{k}(s-\frac{1}{k}),\mathcal{L}(X_{k}(s-\frac{1}{k})))-G(s,0,\delta_{0})\|_{L_{2}^{0}}^{2}ds\nonumber\\
&&+C_{T}M_{1}\int_{0}^{T}\mathbb{E}\|G(s,0,\delta_{0})\|_{L_{2}^{0}}^{2}ds\nonumber\\
&\leq&C_{T}M_{1}\mathbb{E}\int_{0}^{T}K(s)[\Psi(\|X_{k}(s-\frac{1}{k})\|^{2}+\mathcal{W}_{2}(\mathcal{L}(X_{k}(s-\frac{1}{k})),\delta_{0})^{2})+1]ds\nonumber\\
&\leq&C_{T}M_{1}\int_{0}^{T}K(s)(\Psi(2\mathbb{E}\|X_{k}(s-\frac{1}{k})\|^{2})+1)ds.\label{UB5}
\end{eqnarray}
For $I_{6}(t)$, note that Lemma \ref{KN},
\begin{eqnarray}
I_{6}(t)&\leq&6\mathbb{E}\sup_{0\leq t\leq T}\Big\|\int_{0}^{t}\int_{Z}S(t-s)J(s-,X_{k}((s-\frac{1}{k})-),\mathcal{L}(X_{k}((s-\frac{1}{k})-)),z)\tilde{N}(ds,dz)\Big\|^{2}\nonumber\\
&\leq&C_{T}\int_{0}^{T}\int_{Z}\mathbb{E}\|S(t-s)\|^{2}\|J(s-,X_{k}((s-\frac{1}{k})-),\mathcal{L}(X_{k}((s-\frac{1}{k})-)),z)\|^{2}\nu(dz)ds\nonumber\\
&\leq& C_{T}M_{1}\mathbb{E}\int_{0}^{T}\int_{Z}\|J(s-,X_{k}((s-\frac{1}{k})-),\mathcal{L}(X_{k}((s-\frac{1}{k})-)),z)-J(s,0,\delta_{0})\|^{2}\nu(dz)ds\nonumber\\
&&+C_{T}M_{1}\mathbb{E}\int_{0}^{T}\|J(s,0,\delta_{0})\|^{2}\nu(dz)ds\nonumber\\
&\leq&C_{T}M_{1}\mathbb{E}\int_{0}^{T}K(s)[\Psi(\|X_{k}(s-\frac{1}{k})\|^{2}+\mathcal{W}_{2}(\mathcal{L}(X_{k}(s-\frac{1}{k})),\delta_{0})^{2})+1]ds\nonumber\\
&\leq&C_{T}M_{1}\int_{0}^{T}K(s)(\Psi(2\mathbb{E}\|X_{k}(s-\frac{1}{k})\|^{2})+1)ds.\label{UB6}
\end{eqnarray}
Note that $\Psi(\cdot)$ is nondecreasing, concave function, then there exist a positive number $\beta$ such that
\begin{equation}\label{ES}
\Psi(u)\leq \beta(1+u).
\end{equation}
Then, combining (\ref{UB1})-(\ref{UB6}) with (\ref{UB}),
\begin{eqnarray*}
&&\mathbb{E}\sup_{0\leq t\leq T}\|X_{k}(t)\|^{2}\leq 6M_{1}\mathbb{E}\|X_{1}\|^{2}+12M_{1}(1+M_{B}^{2})\mathbb{E}\|X_{0}\|^{2}\\
&&\quad+6M_{B}^{2}M_{1}\int_{0}^{T}\mathbb{E}\|X_{k}(s-\frac{1}{k})\|^{2}ds+C_{T}M_{1}\int_{0}^{T}K(s)(\Psi(2\mathbb{E}\|X_{k}(s-\frac{1}{k})\|^{2})+1)ds\\
&&\leq 6M_{1}\mathbb{E}\|X_{1}\|^{2}+12M_{1}(1+M_{B}^{2})\mathbb{E}\|X_{0}\|^{2}+6M_{B}^{2}M_{1}\int_{0}^{T}\mathbb{E}\|X_{k}(s-\frac{1}{k})\|^{2}ds\\
&&\quad+C_{T}M_{1}\int_{0}^{T}K(s)[\beta(2\mathbb{E}\|X_{k}(s-\frac{1}{k})\|^{2}+1)+1]ds\\
&&\leq 6M_{1}\mathbb{E}\|X_{1}\|^{2}+12M_{1}(1+M_{B}^{2})\mathbb{E}\|X_{0}\|^{2}+6M_{B}^{2}M_{1}\int_{0}^{T}\mathbb{E}\sup_{0\leq s\leq t}\|X_{k}(s)\|^{2}ds\\
&&\quad+C_{T}M_{1}(\beta+1)\int_{0}^{T}K(s)ds+C_{T}M_{1}\beta K(T)\int_{0}^{T}\mathbb{E}\sup_{0\leq s\leq t}\|X_{k}(s)\|^{2}ds\\
&&\leq 6M_{1}\mathbb{E}\|X_{1}\|^{2}+12M_{1}(1+M_{B}^{2})\mathbb{E}\|X_{0}\|^{2}+C_{T}M_{1}(\beta+1)\int_{0}^{T}K(s)ds\\
&&\quad+(C_{T}M_{B}^{2}M_{1}+\beta C_{T}M_{1}K(T))\int_{0}^{T}\mathbb{E}\sup_{0\leq s\leq t}\|X_{k}(s)\|^{2}dt,
\end{eqnarray*}
Gronwall's inequality yields
\begin{eqnarray*}
\mathbb{E}\sup_{0\leq t\leq T}\|X_{k}(t)\|^{2}\leq C_{1},
\end{eqnarray*}
where
\begin{eqnarray*} 
C_{1}&=&[6M_{1}\mathbb{E}\|X_{1}\|^{2}+12M_{1}(1+M_{B}^{2})\mathbb{E}\|X_{0}\|^{2}\\
&&+C_{T}M_{1}(\beta+1)\int_{0}^{T}K(s)ds]e^{6TM_{B}^{2}M_{1}+\beta C_{T}M_{1}K(T)}.
\end{eqnarray*}
\end{proof}
\begin{lemma}
Suppose that $(\mathbf{H_{1}})$-$(\mathbf{H_{3}})$ are valid, then for any $k\geq 1$, $s, t\in[0,T]$, there exists a constant $C_{3}>0$ such that 
\begin{eqnarray*}
\mathbb{E}\|X_{k}(t)-X_{k}(s)\|^{2}\leq C_{3}(|t-s|^{2}+\Big\|S\Big(\frac{t-s}{2}\Big)\Big\|^{2}+|t-s|).
\end{eqnarray*}
\end{lemma}
\begin{proof}
By (\ref{Ca}), 
\begin{eqnarray}\label{EUX}
&&X_{k}(t)-X_{k}(s)=S(t)X_{1}-S(s)X_{1}+[(C(t)-S(t)B)-(C(s)-S(s)B)]X_{0}\nonumber\\
&&\quad\quad\quad+\int_{s}^{t}C(t-r)BX_{k}(r-\frac{1}{k})dr+\int_{0}^{s}[C(t-r)-C(s-r)]BX_{k}(r-\frac{1}{k})dr\nonumber\\
&&\quad\quad\quad+\int_{s}^{t}S(t-r)F(r,X_{k}(r-\frac{1}{k}),\mathcal{L}(X_{k}(r-\frac{1}{k})))dr\nonumber\\
&&\quad\quad\quad+\int_{0}^{s}[S(t-r)-S(s-r)]F(r,X_{k}(r-\frac{1}{k}),\mathcal{L}(X_{k}(r-\frac{1}{k})))dr\nonumber\\
&&\quad\quad\quad+\int_{s}^{t}S(t-r)G(s,X_{k}(r-\frac{1}{k}),\mathcal{L}(X_{k}(r-\frac{1}{k})))dW(r)\nonumber\\
&&\quad\quad\quad+\int_{0}^{s}[S(t-r)-S(s-r)]G(r,X_{k}(r-\frac{1}{k}),\mathcal{L}(X_{k}(r-\frac{1}{k})))dW(r)\nonumber\\
&&\quad\quad\quad+\int_{s}^{t}\int_{Z}S(t-r)J(r-,X_{k}((r-\frac{1}{k})-),\mathcal{L}(X_{k}((r-\frac{1}{k})-)),z)\tilde{N}(dr,dz)\nonumber\\
&&\quad\quad\quad+\int_{0}^{s}\int_{Z}[S(t-r)-S(s-r)]J(r-,X_{k}((r-\frac{1}{k})-),\mathcal{L}(X_{k}((r-\frac{1}{k})-)),z)\tilde{N}(dr,dz)\nonumber\\
&&\quad\quad\quad\triangleq\sum_{i=1}^{10}J_{i}.\label{EUZ}
\end{eqnarray}
For $J_{1}$-$J_{4}$, by Proposition \ref{prCS}, ($\mathbf{H}_{1}$)-($\mathbf{H}_{2}$) and the H${\rm\ddot{o}}$lder inequality, we have
\begin{eqnarray}
\mathbb{E}\|J_{1}\|^{2}&=&\mathbb{E}\|S(t)X_{1}-S(s)X_{1}\|^{2}\leq \mathbb{E}\|S(t)-S(s)\|^{2}\|X_{1}\|^{2}\nonumber\\
&\leq&N_{S}^{2}|t-s|^{2}\mathbb{E}\|X_{1}\|^{2}.\label{EUX1}\\
\mathbb{E}\|J_{2}\|^{2}&=&\mathbb{E}\|[(C(t)-S(t)B)-(C(s)-S(s)B)]X_{0}\|^{2}\nonumber\\
&\leq&2\mathbb{E}\|(C(t)-C(s))X_{0}\|^{2}+2\mathbb{E}\|[S(t)-S(s)]BX_{0}\|^{2}\nonumber\\
&\leq&2\mathbb{E}\Big\|S\Big(\frac{t-s}{2}\Big)\Big\|^{2}\Big\|AS\Big(\frac{t+s}{2}\Big)X_{0}\Big\|^{2}+2M_{B}^{2}N_{S}^{2}|t-s|^{2}\mathbb{E}\|X_{0}\|^{2}\nonumber
\end{eqnarray}
\begin{eqnarray}
&\leq&C_{T}\Big\|S\Big(\frac{t-s}{2}\Big)\Big\|^{2}\mathbb{E}\|X_{0}\|^{2}+2M_{B}^{2}N_{S}^{2}|t-s|^{2}\mathbb{E}\|X_{0}\|^{2}.\\
\mathbb{E}\|J_{3}\|^{2}&\leq&\mathbb{E}\Big(\int_{s}^{t}\|C(t-r)\|\|B\|\|X_{k}(r-\frac{1}{k})\|dr\Big)^{2}\nonumber\\
&\leq& M_{1}M_{B}^{2}|t-s|\int_{s}^{t}\mathbb{E}\|X_{k}(r-\frac{1}{k})\|^{2}dr.\\
\mathbb{E}\|J_{4}\|^{2}&\leq&\mathbb{E}\Big(\int_{0}^{s}\|C(t-r)-C(s-r)\|\|BX_{k}(r-\frac{1}{k})\|dr\Big)^{2}\nonumber\\
&\leq&\Big\|S\Big(\frac{t-s}{2}\Big)\Big\|^{2}\Big(\int_{0}^{s}\Big\|AS\Big(\frac{t+s}{2}-r\Big)\Big\|^{2}\|BX_{k}(r-\frac{1}{k})\|dr\Big)^{2}\nonumber\\
&\leq&C_{T}M_{B}^{2}\Big\|S\Big(\frac{t-s}{2}\Big)\Big\|^{2}\int_{0}^{s}\mathbb{E}\|X_{k}(r-\frac{1}{k})\|^{2}dr.
\end{eqnarray}
For $J_{5}$ and $J_{6}$, ($\mathbf{H}_{1}$), ($\mathbf{H}_{3}$) and the H${\rm\ddot{o}}$lder inequality yield
\begin{eqnarray}
\mathbb{E}\|J_{5}\|^{2}&\leq&2|t-s|\mathbb{E}\int_{s}^{t}\|S(t-r)\|^{2}\|F(r,X_{k}(r-\frac{1}{k}),\mathcal{L}(X_{k}(r-\frac{1}{k})))-F(r,0,\delta_{0})\|^{2}dr\nonumber\\
&&+2|t-s|\mathbb{E}\int_{s}^{t}\|S(t-r)\|^{2}\|F(r,0,\delta_{0})\|^{2}dr\nonumber\\
&\leq&2M_{1}|t-s|\mathbb{E}\int_{s}^{t}K(r)[\Psi(\|X_{k}(r-\frac{1}{k})\|^{2}+\mathcal{W}_{2}(\mathcal{L}(X_{k}(r-\frac{1}{k})),\delta_{0})^{2})+1]dr\nonumber\\
&\leq&2M_{1}|t-s|\int_{s}^{t}K(r)(\Psi(2\mathbb{E}\|X_{k}(r-\frac{1}{k})\|^{2})+1)dr.\\
\mathbb{E}\|J_{6}\|^{2}&\leq&2\mathbb{E}\|\int_{0}^{s}[S(t-r)-S(s-r)]F(r,X_{k}(r-\frac{1}{k}),\mathcal{L}(X_{k}(r-\frac{1}{k})))dr\|^{2}\nonumber\\
&&+2T\mathbb{E}\int_{0}^{s}\|S(t-r)-S(s-r)\|^{2}\|F(r,0,\delta_{0})\|^{2}dr\nonumber\\
&\leq&2TN_{S}^{2}|t-s|^{2}\mathbb{E}\int_{0}^{s}K(r)[\Psi(\|X_{k}(r-\frac{1}{k})\|^{2}+\mathcal{W}_{2}(\mathcal{L}(X_{k}(r-\frac{1}{k})),\delta_{0})^{2})+1]dr\nonumber\\
&\leq&2TN_{S}^{2}|t-s|^{2}\int_{0}^{s}K(r)(\Psi(2\mathbb{E}\|X_{k}(r-\frac{1}{k})\|^{2})+1)dr.
\end{eqnarray}
For $J_{7}$-$J_{10}$, by the ${\rm Burkh\ddot{o}lder}$-Davis-Gundy inequality and Proposition \ref{prCS},
\begin{eqnarray}\label{EJ7}
\mathbb{E}\|J_{7}\|^{2}&\leq&\mathbb{E}\Big\|\int_{s}^{t}S(t-r)G(s,X_{k}(r-\frac{1}{k}),\mathcal{L}(X_{k}(r-\frac{1}{k})))dW(r)\Big\|^{2}\nonumber\\
&\leq&8\mathbb{E}\int_{s}^{t}\|S(t-r)\|^{2}\|G(s,X_{k}(r-\frac{1}{k}),\mathcal{L}(X_{k}(r-\frac{1}{k})))-G(r,0,\delta_{0})\|_{L_{2}^{0}}^{2}dr\nonumber\\
&&+8\mathbb{E}\int_{s}^{t}\|S(t-r)\|^{2}\|G(r,0,\delta_{0})\|_{L_{2}^{0}}^{2}dr\nonumber\\
&\leq&8M_{1}\mathbb{E}\int_{s}^{t}K(r)[\Psi(\|X_{k}(r-\frac{1}{k})\|^{2}+\mathcal{W}_{2}(\mathcal{L}(X_{k}(r-\frac{1}{k})),\delta_{0})^{2})+1]dr\nonumber
\end{eqnarray}
\begin{eqnarray}
&\leq&8M_{1}\int_{s}^{t}K(r)(\Psi(2\mathbb{E}\|X_{k}(r-\frac{1}{k})\|^{2})+1)dr.\\
\mathbb{E}\|J_{8}\|^{2}&\leq&\mathbb{E}\Big\|\int_{0}^{s}[S(t-r)-S(s-r)]G(r,X_{k}(r-\frac{1}{k}),\mathcal{L}(X_{k}(r-\frac{1}{k})))dW(r)\Big\|^{2}\nonumber\\
&\leq&8\mathbb{E}\int_{0}^{s}\|S(t-r)-S(s-r)\|^{2}\|G(r,X_{k}(r-\frac{1}{k}),\mathcal{L}(X_{k}(r-\frac{1}{k})))-G(r,0,\delta_{0})\|_{L_{2}^{0}}^{2}dr\nonumber\\
&&+8\mathbb{E}\int_{0}^{s}\|S(t-r)-S(s-r)\|^{2}\|G(r,0,\delta_{0})\|_{L_{2}^{0}}^{2}dr\nonumber\\
&\leq&8N_{S}^{2}|t-s|^{2}\mathbb{E}\int_{0}^{s}K(r)[\Psi(\|X_{k}(r-\frac{1}{k})\|^{2}+\mathcal{W}_{2}(\mathcal{L}(X_{k}(r-\frac{1}{k})),\delta_{0})^{2})+1]dr\nonumber\\
&\leq&8N_{S}^{2}|t-s|^{2}\int_{0}^{s}K(r)(\Psi(2\mathbb{E}\|X_{k}(r-\frac{1}{k})\|^{2})+1)dr.
\end{eqnarray}
\begin{eqnarray}
\mathbb{E}\|J_{9}\|^{2}&\leq&\mathbb{E}\Big\|\int_{s}^{t}\int_{Z}S(t-r)J(r-,X_{k}((r-\frac{1}{k})-),\mathcal{L}(X_{k}((r-\frac{1}{k})-)),z)\tilde{N}(dr,dz)\Big\|^{2}\nonumber\\
&\leq&C_{T}\mathbb{E}\int_{s}^{t}\int_{Z}\|S(t-r)\|^{2}\|J(r-,X_{k}((r-\frac{1}{k})-),\mathcal{L}(X_{k}((r-\frac{1}{k})-)),z)\nonumber\\
&&~~-J(r-,0,\delta_{0},z)\|^{2}\nu(dz)dr\nonumber\\
&&+C_{T}\mathbb{E}\int_{s}^{t}\int_{Z}\|S(t-r)\|^{2}\|J(r-,0,\delta_{0},z)\|^{2}\nu(dz)dr\nonumber\\
&\leq&C_{T}M_{1}\mathbb{E}\int_{s}^{t}K(r)[\Psi(\|X_{k}(r-\frac{1}{k})\|^{2}+\mathcal{W}_{2}(\mathcal{L}(X_{k}(r-\frac{1}{k})),\delta_{0})^{2})+1]dr\nonumber\\
&\leq&C_{T}M_{1}\int_{s}^{t}K(r)(\Psi(2\mathbb{E}\|X_{k}(r-\frac{1}{k})\|^{2})+1)dr.\\
\mathbb{E}\|J_{10}\|^{2}&\leq&\mathbb{E}\Big\|\int_{0}^{s}\int_{Z}[S(t-r)-S(s-r)]J(r-,X_{k}((r-\frac{1}{k})-),\mathcal{L}(X_{k}((r-\frac{1}{k})-)),z)\tilde{N}(dr,dz)\Big\|^{2}\nonumber\\
&\leq&C_{T}M_{1}|t-s|^{2}\mathbb{E}\int_{0}^{s}\int_{Z}\|J(r-,X_{k}((r-\frac{1}{k})-),\mathcal{L}(X_{k}((r-\frac{1}{k})-)),z)\nonumber\\
&&~~-J(r-,0,\delta_{0},z)\|^{2}\nu(dz)dr\nonumber\\
&&+C_{T}M_{1}|t-s|^{2}\mathbb{E}\int_{0}^{s}\int_{Z}\|J(r-,0,\delta_{0},z)\|^{2}\nu(dz)dr\nonumber\\
&\leq&C_{T}M_{1}|t-s|^{2}\mathbb{E}\int_{0}^{s}K(r)[\Psi(\|X_{k}(r-\frac{1}{k})\|^{2}+\mathcal{W}_{2}(\mathcal{L}(X_{k}(r-\frac{1}{k})),\delta_{0})^{2})+1]dr\nonumber\\
&\leq&C_{T}M_{1}|t-s|^{2}\int_{s}^{t}K(r)(\Psi(2\mathbb{E}\|X_{k}(r-\frac{1}{k})\|^{2})+1)dr.\label{EUX10}
\end{eqnarray}
Then combining (\ref{EUX}) with (\ref{EUX1})-(\ref{EUX10}) and by Lemma \ref{SUB}, we have
\begin{eqnarray*}
&&\mathbb{E}\|X_{k}(t)-X_{k}(s)\|^{2}\leq[10N_{S}^{2}\mathbb{E}\|X_{1}\|^{2}+20M_{B}^{2}N_{S}^{2}\mathbb{E}\|X_{0}\|^{2}]|t-s|^{2}\\
&&\quad+C_{T}\Big\|S\Big(\frac{t-s}{2}\Big)\Big\|^{2}\mathbb{E}\|X_{0}\|^{2}+10M_{1}M_{B}^{2}|t-s|\int_{s}^{t}\mathbb{E}\|X_{k}(r-\frac{1}{k})\|^{2}dr
\end{eqnarray*}
\begin{eqnarray*}
&&\quad+C_{T}M_{B}^{2}\Big\|S\Big(\frac{t-s}{2}\Big)\Big\|^{2}\int_{0}^{s}\mathbb{E}\|X_{k}(r-\frac{1}{k})\|^{2}dr\\
&&\quad+20M_{1}|t-s|\int_{s}^{t}K(r)(\Psi(2\mathbb{E}\|X_{k}(r-\frac{1}{k})\|^{2})+1)dr\\
&&\quad+20TN_{S}^{2}|t-s|^{2}\int_{0}^{s}K(r)(\Psi(2\mathbb{E}\|X_{k}(r-\frac{1}{k})\|^{2})+1)dr\\
&&\quad+80M_{1}\int_{s}^{t}K(r)(\Psi(2\mathbb{E}\|X_{k}(r-\frac{1}{k})\|^{2})+1)dr\\
&&\quad+80N_{S}^{2}|t-s|^{2}\int_{0}^{s}K(r)(\Psi(2\mathbb{E}\|X_{k}(r-\frac{1}{k})\|^{2})+1)dr\\
&&\quad+C_{T}M_{1}\int_{s}^{t}K(r)(\Psi(2\mathbb{E}\|X_{k}(r-\frac{1}{k})\|^{2})+1)dr\\
&&\quad+C_{T}M_{1}|t-s|^{2}\int_{s}^{t}K(r)(\Psi(2\mathbb{E}\|X_{k}(r-\frac{1}{k})\|^{2})+1)dr\\
&&\leq[10N_{S}^{2}\mathbb{E}\|X_{1}\|^{2}+20M_{B}^{2}N_{S}^{2}\mathbb{E}\|X_{0}\|^{2}]|t-s|^{2}+C_{T}\Big\|S\Big(\frac{t-s}{2}\Big)\Big\|^{2}\mathbb{E}\|X_{0}\|^{2}\\
&&\quad+10M_{1}M_{B}^{2}\mathbb{E}\sup_{0\leq t\leq T}\|X_{k}(t)\|^{2}|t-s|^{2}\quad+C_{T}M_{B}^{2}\Big\|S\Big(\frac{t-s}{2}\Big)\Big\|^{2}\mathbb{E}\sup_{0\leq t\leq T}\|X_{k}(t)\|^{2}\\
&&\quad+20M_{1}K(T)[2\beta(\mathbb{E}\sup_{0\leq t\leq T}\|X_{k}(t)\|^{2}+1)+1]|t-s|^{2}\\
&&\quad+20T^{2}N_{S}^{2}K(T)|t-s|^{2}[2\beta(\mathbb{E}\sup_{0\leq t\leq T}\|X_{k}(t)\|^{2}+1)+1]\\
&&\quad+80M_{1}K(T)|t-s|[2\beta(\mathbb{E}\sup_{0\leq t\leq T}\|X_{k}(t)\|^{2}+1)+1]\\
&&\quad+80N_{S}^{2}|t-s|^{2}K(T)T[2\beta(\mathbb{E}\sup_{0\leq t\leq T}\|X_{k}(t)\|^{2}+1)+1]\\
&&\quad+C_{T}M_{1}K(T)|t-s|[2\beta(\mathbb{E}\sup_{0\leq t\leq T}\|X_{k}(t)\|^{2}+1)+1]\\
&&\quad+C_{T}M_{1}K(T)T|t-s|^{2}[2\beta(\mathbb{E}\sup_{0\leq t\leq T}\|X_{k}(t)\|^{2}+1)+1]\\
&&\triangleq C_{2}(|t-s|^{2}+\Big\|S\Big(\frac{t-s}{2}\Big)\Big\|^{2}+|t-s|).
\end{eqnarray*}
Here
\begin{eqnarray*}
C_{2}&=&10N_{S}^{2}\mathbb{E}\|X_{1}\|^{2}+20M_{B}^{2}N_{S}^{2}\mathbb{E}\|X_{0}\|^{2}+10M_{1}M_{B}^{2}C_{1}\\
&&+[20M_{1}K(T)+20T^{2}N_{S}^{2}K(T)+80N_{S}^{2}TK(T)+C_{T}M_{1}K(T)][2\beta(C_{1}+1)+1]\\
&&+C_{T}\mathbb{E}\|X_{0}\|^{2}+C_{T}M_{B}^{2}C_{1}+[80M_{1}K(T)+C_{T}M_{1}K(T)][2\beta(C_{1}+1)+1].
\end{eqnarray*}
\end{proof}
Next we prove the main theorem of this paper.
\begin{theorem}\label{mainthm}
Suppose that $(\mathbf{H}_{1})$-$(\mathbf{H}_{3})$ hold and $X_{0}\in L^{2}(\Omega; E)$, $X_{1}\in L^{2}(\Omega; H)$, then the equation (\ref{main}) has a unique mild solution $X(t)\in L^{2}(\Omega; D([0,T]; H))$.
\end{theorem}
\begin{proof}
We divide the proof into two steps.

\textbf{Step 1.}~(Existence) we prove that $(X_{k})_{k\geq1}$ is a Cauchy sequence in $L^{2}(\Omega;D([0,T];H))$. In fact,
for any $m>n\geq1$, we have
\begin{eqnarray}\label{ER}
&&\sup_{0\leq t\leq T}\|X_{m}(s)-X_{n}(s)\|^{2}\nonumber\\
&&\leq 4\sup_{0\leq s\leq t}\Big\|\int_{0}^{s}C(s-u)B[X_{m}(u-\frac{1}{m})-X_{n}(u-\frac{1}{n})]du\Big\|^{2}\nonumber\\
&&+4\sup_{0\leq s\leq t}\Big\|\int_{0}^{s}S(s-u)[F(u,X_{m}(u-\frac{1}{m}),\mathcal{L}(X_{m}(u-\frac{1}{m}))\nonumber\\
&&~~-F(u,X_{n}(u-\frac{1}{n}),\mathcal{L}(X_{n}(u-\frac{1}{n}))]du\Big\|^{2}\nonumber\\
&&+4\sup_{0\leq s\leq t}\Big\|\int_{0}^{s}S(s-u)[G(u,X_{m}(u-\frac{1}{m}),\mathcal{L}(X_{m}(u-\frac{1}{m}))\nonumber\\
&&~~-G(u,X_{n}(u-\frac{1}{n}),\mathcal{L}(X_{n}(u-\frac{1}{n}))]dW(u)\Big\|^{2}\nonumber\\
&&+4\sup_{0\leq s\leq t}\Big\|\int_{0}^{s}S(s-u)[J(u-,X_{m}((u-\frac{1}{m})-),\mathcal{L}(X_{m}((u-\frac{1}{m})-)),z)\nonumber\\
&&~~-J(u-,X_{n}((u-\frac{1}{n})-),\mathcal{L}(X_{n}((u-\frac{1}{n})-),z)]\tilde{N}(du,dz)\Big\|^{2}\nonumber\\
&&\triangleq\sum_{i=1}^{4}R_{i}.
\end{eqnarray}
Next we estimate these terms respectively. In fact, by ($\mathbf{H}_{1}$)-($\mathbf{H}_{2}$) and the H${\rm\ddot{o}}$lder inequality,
\begin{eqnarray}
\mathbb{E}R_{1}&=&4\mathbb{E}\sup_{0\leq s\leq t}\Big\|\int_{0}^{s}C(s-u)B[X_{m}(u-\frac{1}{m})-X_{n}(u-\frac{1}{n})]du\Big\|^{2}\nonumber\\
&\leq&8\mathbb{E}\sup_{0\leq s\leq t}\Big\|\int_{0}^{s}C(s-u)B[X_{m}(u-\frac{1}{m})-X_{n}(u-\frac{1}{m})]du\Big\|^{2}\nonumber\\
&&+8\mathbb{E}\sup_{0\leq s\leq t}\Big\|\int_{0}^{s}C(s-u)B[X_{n}(u-\frac{1}{m})-X_{n}(u-\frac{1}{n})]du\Big\|^{2}\nonumber\\
&\leq&8M_{1}M_{B}^{2}\int_{0}^{t}\mathbb{E}\|X_{m}(u-\frac{1}{m})-X_{n}(u-\frac{1}{m})\|^{2}du\nonumber\\
&&+8M_{1}M_{B}^{2}\int_{0}^{t}\mathbb{E}\|X_{n}(u-\frac{1}{m})-X_{n}(u-\frac{1}{n})\|^{2}du\nonumber\\
&\leq&8M_{1}M_{B}^{2}\int_{0}^{t}\mathbb{E}\|X_{m}(u-\frac{1}{m})-X_{n}(u-\frac{1}{m})\|^{2}du\label{ER1}\\
&&+8C_{2}M_{1}M_{B}^{2}\int_{0}^{t}\Big((\frac{1}{m}-\frac{1}{n})^{2}+\Big\|S(\frac{1}{2}(\frac{1}{n}-\frac{1}{m}))\Big\|^{2}+(\frac{1}{n}-\frac{1}{m})\Big)du.\nonumber\\
\mathbb{E}R_{2}&=&4\mathbb{E}\sup_{0\leq s\leq t}\Big\|\int_{0}^{s}S(s-u)[F(u,X_{m}(u-\frac{1}{m}),\mathcal{L}(X_{m}(u-\frac{1}{m})))\nonumber\\
&&~~-F(u,X_{n}(u-\frac{1}{n}),\mathcal{L}(X_{n}(u-\frac{1}{n})))]du\Big\|^{2}\nonumber
\end{eqnarray} 
\begin{eqnarray}
&\leq&8\mathbb{E}\sup_{0\leq s\leq t}\Big\|\int_{0}^{s}S(s-u)[F(u,X_{m}(u-\frac{1}{m}),\mathcal{L}(X_{m}(u-\frac{1}{m})))\nonumber\\
&&~~-F(u,X_{n}(u-\frac{1}{m}),\mathcal{L}(X_{n}(u-\frac{1}{m})))]du\Big\|^{2}\nonumber\\
&&+8\mathbb{E}\sup_{0\leq s\leq t}\Big\|\int_{0}^{s}S(s-u)[F(u,X_{n}(u-\frac{1}{m}),\mathcal{L}(X_{n}(u-\frac{1}{m})))\nonumber\\
&&~~-F(u,X_{n}(u-\frac{1}{n}),\mathcal{L}(X_{n}(u-\frac{1}{n})))]du\Big\|^{2}\nonumber\\
&\leq&8t\mathbb{E}\sup_{0\leq s\leq t}\int_{0}^{s}\|S(s-u)\|^{2}\|F(u,X_{m}(u-\frac{1}{m}),\mathcal{L}(X_{m}(u-\frac{1}{m})))\nonumber\\
&&~~-F(u,X_{n}(u-\frac{1}{m}),\mathcal{L}(X_{n}(u-\frac{1}{m})))\|^{2}du\nonumber\\
&&+8t\mathbb{E}\sup_{0\leq s\leq t}\int_{0}^{s}\|S(s-u)\|^{2}\|F(u,X_{n}(u-\frac{1}{m}),\mathcal{L}(X_{n}(u-\frac{1}{m})))\nonumber\\
&&~~-F(u,X_{n}(u-\frac{1}{n}),\mathcal{L}(X_{n}(u-\frac{1}{n})))\|^{2}du\nonumber\\
&\leq&8tM_{1}\int_{0}^{t}K(u)\Psi(2\mathbb{E}\|X_{m}(u-\frac{1}{m})-X_{n}(u-\frac{1}{m})\|^{2})du\label{ER2}\\
&&+8tM_{1}\int_{0}^{t}K(u)\Psi(2C_{2}(\frac{1}{n}-\frac{1}{m})^{2}+\|S(\frac{1}{2}(\frac{1}{n}-\frac{1}{m}))\|^{2}+(\frac{1}{n}-\frac{1}{m}))du.\nonumber
\end{eqnarray}
Similar to (\ref{EJ7}), by the ${\rm Burkh\ddot{o}lder}$-Davis-Gundy inequality, ($\mathbf{H}_{1}$) and ($\mathbf{H}_{3}$),
\begin{eqnarray}\label{ER16}
\mathbb{E}R_{3}&\leq&8\mathbb{E}\sup_{0\leq s\leq t}\Big\|\int_{0}^{s}S(s-u)[G(u,X_{m}(u-\frac{1}{m}),\mathcal{L}(X_{m}(u-\frac{1}{m})))\nonumber\\
&&~~-G(u,X_{n}(u-\frac{1}{m}),\mathcal{L}(X_{n}(u-\frac{1}{m})))]dW(u)\Big\|^{2}\nonumber\\
&&+8\mathbb{E}\sup_{0\leq s\leq t}\Big\|\int_{0}^{s}S(s-u)[G(u,X_{n}(u-\frac{1}{m}),\mathcal{L}(X_{n}(u-\frac{1}{m})))\nonumber\\
&&~~-G(u,X_{n}(u-\frac{1}{n}),\mathcal{L}(X_{n}(u-\frac{1}{n})))]dW(u)\Big\|^{2}\nonumber\\
&\leq&8M_{1}\int_{0}^{t}K(u)\Psi(2\mathbb{E}\|X_{m}(u-\frac{1}{m})-X_{n}(u-\frac{1}{m})\|^{2})du\\
&&+8M_{1}\int_{0}^{t}K(u)\Psi(2C_{2}(\frac{1}{n}-\frac{1}{m})^{2}+\|S(\frac{1}{2}(\frac{1}{n}-\frac{1}{m}))\|^{2}+(\frac{1}{n}-\frac{1}{m}))du.\nonumber
\end{eqnarray}
Lastly, by Lemma \ref{KN}, ($\mathbf{H}_{1}$) and ($\mathbf{H}_{3}$), we have
\begin{eqnarray}
\mathbb{E}R_{4}&\leq&8\mathbb{E}\sup_{0\leq s\leq t}\Big\|\int_{0}^{s}S(s-u)[J(u-,X_{m}((u-\frac{1}{m})-),\mathcal{L}(X_{m}((u-\frac{1}{m})-)),z)\nonumber\\
&&~~-J(u-,X_{n}((u-\frac{1}{m})-),\mathcal{L}(X_{n}((u-\frac{1}{m})-)),z)]\tilde{N}(du,dz)\Big\|^{2}\nonumber\\
&&+8\mathbb{E}\sup_{0\leq s\leq t}\Big\|\int_{0}^{s}S(s-u)[J(u-,X_{n}((u-\frac{1}{m})-),\mathcal{L}(X_{n}((u-\frac{1}{m})-)),z)\nonumber
\end{eqnarray}
\begin{eqnarray}
&&~~-J(u-,X_{n}((u-\frac{1}{n})-),\mathcal{L}(X_{n}((u-\frac{1}{n})-)),z)]\tilde{N}(du,dz)\Big\|^{2}\nonumber\\
&\leq&8CM_{1}\int_{0}^{t}\int_{Z}\mathbb{E}\|J(u-,X_{m}((u-\frac{1}{m})-),\mathcal{L}(X_{m}((u-\frac{1}{m})-)),z)\nonumber\\
&&~~-J(u-,X_{n}((u-\frac{1}{m})-),\mathcal{L}(X_{n}((u-\frac{1}{m})-)),z)\|^{2}\nu(dz)du\nonumber
\end{eqnarray} 
\begin{eqnarray}\label{ER17}
&&+8CM_{1}\int_{0}^{t}\int_{Z}\mathbb{E}\|J(u-,X_{n}((u-\frac{1}{m})-),\mathcal{L}((X_{n}(u-\frac{1}{m})-)),z)\nonumber\\
&&~~-J(u-,X_{n}((u-\frac{1}{n})-),\mathcal{L}(X_{n}((u-\frac{1}{n})-)),z)\|^{2}\nu(dz)du\nonumber\\
&\leq&8CM_{1}\int_{0}^{t}K(u)\Psi(2\mathbb{E}\|X_{m}(u-\frac{1}{m})-X_{n}(u-\frac{1}{m})\|^{2})du\label{ER4}\\
&&+8CM_{1}\int_{0}^{t}K(u)\Psi(2C_{2}(\frac{1}{n}-\frac{1}{m})^{2}+\|S(\frac{1}{2}(\frac{1}{n}-\frac{1}{m}))\|^{2}+(\frac{1}{n}-\frac{1}{m}))du.\nonumber
\end{eqnarray} 
Substituting (\ref{ER1})-(\ref{ER4}) into (\ref{ER}), we have
\begin{eqnarray*}
&&\mathbb{E}\sup_{0\leq u\leq t}\|X_{m}(u)-X_{n}(u)\|^{2}\leq8M_{1}M_{B}^{2}\int_{0}^{t}\mathbb{E}\sup_{0\leq v\leq u}\|X_{m}(v)-X_{n}(v)\|^{2}du\\
&&~~~~+(16TM_{1}+8CM_{1})K(T)\int_{0}^{t}\Psi(2\mathbb{E}\|X_{m}(u)-X_{n}(u)\|^{2})du\\
&&~~~~+8C_{2}M_{1}M_{B}^{2}T((\frac{1}{n}-\frac{1}{m})^{2}+\|S(\frac{1}{2}(\frac{1}{n}-\frac{1}{m}))\|^{2}+(\frac{1}{n}-\frac{1}{m}))\\
&&~~~~+16TM_{1}^{2}K(T)\Psi(2C_{2}(\frac{1}{n}-\frac{1}{m})^{2}+\|S(\frac{1}{2}(\frac{1}{n}-\frac{1}{m}))\|^{2}+(\frac{1}{n}-\frac{1}{m}))\\
&&~~~~+8CM_{1}TK(T)\Psi(2C_{2}(\frac{1}{n}-\frac{1}{m})^{2}+\|S(\frac{1}{2}(\frac{1}{n}-\frac{1}{m}))\|^{2}+(\frac{1}{n}-\frac{1}{m})).
\end{eqnarray*}
Let $Z(t)=\lim_{m,n\to\infty}\mathbb{E}\sup_{0\leq s\leq t}\|X_{m}(s)-X_{n}(s)\|^{2}$ and $\bar{\Psi}(u)=\Psi(u)+u$, then $\bar{\Psi}: [0,\infty)\to[0,\infty)$ is continuous, nondecreasing and concave function and $\bar{\Psi}(0)=0, \bar{\Psi}(x)>0$ for any $x>0$, furthermore, $\int_{0^{+}}\frac{1}{\bar{\Psi}(x)}dx=+\infty$. Then
\begin{eqnarray*}
&&\mathbb{E}\sup_{0\leq u\leq t}\|X_{m}(u)-X_{n}(u)\|^{2}\\
&&\leq (4M_{1}M_{B}^{2}+16TM_{1}+8CM_{1}K(T))\int_{0}^{t}\bar{\Psi}(2\mathbb{E}\sup_{0\leq v\leq u}\|X_{m}(v)-X_{n}(v)\|^{2})du\\
&&+8C_{2}M_{1}M_{B}^{2}T((\frac{1}{n}-\frac{1}{m})^{2}+\|S(\frac{1}{2}(\frac{1}{n}-\frac{1}{m}))\|^{2}+(\frac{1}{n}-\frac{1}{m}))\\
&&+16TM_{1}^{2}K(T)\Psi(2C_{2}(\frac{1}{n}-\frac{1}{m})^{2}+\|S(\frac{1}{2}(\frac{1}{n}-\frac{1}{m}))\|^{2}+(\frac{1}{n}-\frac{1}{m}))\\
&&+8CM_{1}TK(T)\Psi(2C_{2}(\frac{1}{n}-\frac{1}{m})^{2}+\|S(\frac{1}{2}(\frac{1}{n}-\frac{1}{m}))\|^{2}+(\frac{1}{n}-\frac{1}{m})).\end{eqnarray*}
Let $m,n\to\infty$, and by $\Psi(0)=0$, for any $\epsilon>0$, we have
\begin{eqnarray*}
Z(t)\leq \epsilon+(4M_{1}M_{B}^{2}+16TM_{1}+8CM_{1}K(T))\int_{0}^{t}\bar{\Psi}(2Z(s))ds.
\end{eqnarray*}
Bihari's inequality yields
\begin{eqnarray*}
Z(t)\leq \frac{1}{2}G^{-1}[G(2\epsilon)+2(4M_{1}M_{B}^{2}+16TM_{1}+8CM_{1}K(T))t],
\end{eqnarray*}
where $G(2\epsilon)+2(4M_{1}M_{B}^{2}+16TM_{1}+8CM_{1}K(T))t\in Dom(G^{-1})$ and $G^{-1}$ is the inverse of $G$ with $G(v)=\int_{1}^{v}\frac{ds}{\bar{\Psi}(s)}$, $v>0$. By the assumption $(\mathbf{H}_{3})$, we have $\lim_{\epsilon\to0}G(\epsilon)=-\infty$ and $Dom(G^{-1})=(-\infty, G(\infty))$. Let $\epsilon\to0$, we obtain $Z(t)=0$, that is
\begin{eqnarray*}
\mathbb{E}\sup_{0\leq t\leq T}\|X_{n}(s)-X_{m}(s)\|^{2}=0,\quad m, n\to\infty.
\end{eqnarray*} 
Hence, $(X_{m})_{m\geq 1}$ is a Cauchy sequence and let $m\to\infty$, we have
\begin{eqnarray*}
\lim_{n\to\infty}\mathbb{E}\sup_{0\leq t\leq T}\|X(t)-X_{n}(t)\|^{2}=0.
\end{eqnarray*}
Now, we prove that $X(t)$ is the solution of equation (\ref{main}). In fact, for any $0\leq t\leq T$,we have
\begin{eqnarray*}
\mathbb{E}\|X(t)-X_{k}(t-\frac{1}{k})\|^{2}\leq 2\mathbb{E}\|X(t)-X_{k}(t)\|^{2}+2\mathbb{E}\|X_{k}(t)-X_{k}(t-\frac{1}{k})\|^{2},
\end{eqnarray*}
then
\begin{eqnarray*}
\mathbb{E}\|X(t)-X_{k}(t-\frac{1}{k})\|^{2}\to0,\quad k\to\infty.
\end{eqnarray*}
Therefore, let $k\to\infty$ on both sides of (\ref{Ca}), we have
\begin{eqnarray*}
X(t)&=&S(t)X_{1}+(C(t)-S(t)B)X_{0}+\int_{0}^{t}C(t-s)BX(s)ds\\
&&+\int_{0}^{t}S(t-s)F(s,X(s),\mathcal{L}(X(s)))ds+\int_{0}^{t}S(t-s)G(s,X(s),\mathcal{L}(X(s)))dW(s)\\
&&+\int_{0}^{t}\int_{Z}S(t-s)J(s-,X(s-),\mathcal{L}(X(s-)),z)\tilde{N}(ds,dz),
\end{eqnarray*}
that is, $X(t)$ is the mild solution of the equation (\ref{main}).

\textbf{Step 2.} (Uniqueness) Assume that $X(t)$ and $Y(t)$ are solutions for equation (\ref{main}) on the same probability space with $X(0)=Y(0), X'(0)=Y'(0)$. Then by the H${\rm \ddot{o}}$lder inequality, the BurkH${\rm\ddot{o}}$lder-Davies-Gundy inequality and $(\mathbf{H}_{1})$-$(\mathbf{H}_{3})$,
\begin{eqnarray*}
&&\mathbb{E}\sup_{0\leq s\leq t}\|X(s)-Y(s)\|^{2}\\
&\leq& 4\mathbb{E}\sup_{0\leq s\leq t}\Big\|\int_{0}^{s}C(s-u)B[X(u)-Y(u)]du\Big\|^{2}\\
&&+4\mathbb{E}\sup_{0\leq s\leq t}\Big\|\int_{0}^{s}S(s-u)[F(u,X(u),\mathcal{L}(X(u)))-F(u,Y(u),\mathcal{L}(Y(u)))]ds\Big\|^{2}\\
&&+4\mathbb{E}\sup_{0\leq s\leq t}\Big\|\int_{0}^{s}S(s-u)[G(u,X(u),\mathcal{L}(X(u)))-G(u,Y(u),\mathcal{L}(Y(u)))]dW(u)\Big\|^{2}
\end{eqnarray*}
\begin{eqnarray*}
&&+4\mathbb{E}\sup_{0\leq s\leq t}\Big\|\int_{0}^{s}\int_{Z}S(s-u)[J(u-,X(u-),\mathcal{L}(X(u-)),z)\\
&&~~-J(u-,Y(u-),\mathcal{L}(Y(u-)),z)]\tilde{N}(ds,dz)\Big\|^{2}\\
&\leq& 4M_{1}M_{B}^{2}T\int_{0}^{t}\mathbb{E}\|X(u)-Y(u)\|^{2}du+4tM_{1}\int_{0}^{t}K(u)\Psi(2\mathbb{E}\sup_{0\leq v\leq u}\|X(v)-Y(v)\|^{2})du\\
&&+16M_{1}\int_{0}^{t}K(u)\Psi(2\mathbb{E}\sup_{0\leq v\leq u}\|X(v)-Y(v)\|^{2})du\\
&&+4CM_{1}\int_{0}^{t}K(u)\Psi(2\mathbb{E}\sup_{0\leq v\leq u}\|X(v)-Y(v)\|^{2})du,
\end{eqnarray*}
then
\begin{eqnarray*}
\mathbb{E}\sup_{0\leq s\leq t}\|X(s)-Y(s)\|^{2}&\leq&(2M_{1}M_{B}^{2}T+4M_{1}K(T)T+16M_{1}K(T)+4CM_{1}K(T))\\
&&\cdot\int_{0}^{t}\bar{\Psi}(2\mathbb{E}\sup_{0\leq v\leq u}\|X(v)-Y(v)\|^{2})du.
\end{eqnarray*}
By the Bihari inequality, we have $X(t)=Y(t), t\in [0,T], \mathbb{P}$-a.s..
\end{proof}

\section{Averaging principle}\label{Aver}
  \setcounter{equation}{0}
  \renewcommand{\theequation}
{4.\arabic{equation}}
In the section, we are devoted to establish an averaging principle for (\ref{main}). The standard integral formulation of equation (\ref{main}) is 
\begin{eqnarray}\label{SX}
X^{\epsilon}(t)&=&S(t)X_{1}+[C(t)-S(t)B]X_{0}+\int_{0}^{t}C(t-s)BX^{\epsilon}(s)ds\nonumber\\
&&+\epsilon\int_{0}^{t}S(t-s)F(s,X^{\epsilon}(s),\mathcal{L}(X^{\epsilon}(s)))ds\nonumber\\
&&+\sqrt{\epsilon}\int_{0}^{t}S(t-s)G(s,X^{\epsilon}(s),\mathcal{L}(X^{\epsilon}(s)))dW(s)\\
&&+\sqrt{\epsilon}\int_{0}^{t}\int_{Z}S(t-s)J(s-,X^{\epsilon}(s-),\mathcal{L}(X^{\epsilon}(s-)),z)\tilde{N}(ds,dz),\quad a.s.\nonumber
\end{eqnarray}
where $0\leq t\leq T$ and $\mathcal{L}(X(t))$ is the distribution of $X(t)$. $\epsilon\in(0,\epsilon_{1}](0<\epsilon_{1}\ll1)$ is a small parameter, the coefficients satisfy the assumptions of Theorem \ref{mainthm}.
Next we show that as $\epsilon\to0$, the standard mild solution $X^{\epsilon}(t)$ may be approximated by the mild solution $Z^{\epsilon}(t)$ of the averaged equation
\begin{eqnarray}
Z^{\epsilon}(t)&=&S(t)X_{1}+[C(t)-S(t)B]X_{0}+\int_{0}^{t}C(t-s)BZ^{\epsilon}(s)ds\nonumber
\end{eqnarray}
\begin{eqnarray}\label{SZ}
&&+\epsilon\int_{0}^{t}S(t-s)\bar{F}(Z^{\epsilon}(s),\mathcal{L}(Z^{\epsilon}(s)))ds\nonumber\\
&&+\sqrt{\epsilon}\int_{0}^{t}S(t-s)\bar{G}(Z^{\epsilon}(s),\mathcal{L}(Z^{\epsilon}(s)))dW(s)\\
&&+\sqrt{\epsilon}\int_{0}^{t}\int_{Z}S(t-s)\bar{J}(Z^{\epsilon}(s-),\mathcal{L}(Z^{\epsilon}(s-)),z)\tilde{N}(ds,dz),\quad a.s.\nonumber
\end{eqnarray}
the coefficients $\bar{F}(\cdot), \bar{G}(\cdot), \bar{J}(\cdot,z)\in H$ are measurable functions satisfying assumptions ($\mathbf{H}_{1}$)-($\mathbf{H}_{3}$), then the Theorem \ref{mainthm} is valid for $Z^{\epsilon}(t)$. Furthermore, the following assumptions hold

($\mathbf{H}_{4}$) For any $T_{1}\in[0,T]$,\\
(i)~\begin{eqnarray*}
\frac{1}{T_{1}}\int_{0}^{T_{1}}\|F(t,\xi,\mu)-\bar{F}(\xi, \mu)\|^{2}dt\leq \phi_{1}(T_{1})\psi(\|\xi\|^{2}+\mathcal{W}_{2}(\mu,\delta_{0})^{2}),
\end{eqnarray*}
(ii)
\begin{eqnarray*}
\frac{1}{T_{1}}\int_{0}^{T_{1}}\|G(t,\xi,\mu)-\bar{G}(\xi, \mu)\|_{L_{2}^{0}}^{2}dt\leq \phi_{2}(T_{1})\psi(\|\xi\|^{2}+\mathcal{W}_{2}(\mu,\delta_{0})^{2}),
\end{eqnarray*}
(iii)
\begin{eqnarray*}
\frac{1}{T_{1}}\int_{0}^{T_{1}}\int_{Z}\|J(t,\xi,\mu,z)-\bar{J}(\xi,\mu,z)\|^{2}\nu(dz)dt\leq \phi_{3}(T_{1})\psi(\|\xi\|^{2}+\mathcal{W}_{2}(\mu,\delta_{0})^{2}),
\end{eqnarray*}
where $\phi_{i}(\cdot)$ are bounded positive functions with $\lim_{T_{1}\to\infty}\phi_{i}(T_{1})=0, i=1,2,3$ and $\psi(\cdot)$ is a continuous, nondecreasing concave function. Moreover, $\psi(0)=0, \psi(x)>0$ for $x>0$, $\int_{0^{+}}\frac{dx}{\psi(x)}=\infty$.
\begin{theorem}
Suppose that {\rm($\mathbf{H}_{1}$)}-{\rm($\mathbf{H}_{4}$)} hold, For any fixed small parameter $\delta_{1}>0$, there exists $L>0, \alpha\in(0,1)$, $\epsilon_{0}>0$ and $\epsilon_{1}\in(0,\epsilon_{0}]$ such that for any $\epsilon\in(0,\epsilon_{1}]$,
\begin{eqnarray*}
\mathbb{E}\sup_{t\in[0,L\epsilon^{-\alpha}]}\|X^{\epsilon}(s)-Z^{\epsilon}(s)\|^{2}&\leq&\delta_{1}.
\end{eqnarray*}
\end{theorem}
\begin{proof}
By (\ref{SX}), (\ref{SZ}), Lemma \ref{ME}, {\rm($\mathbf{H}_{1}$)} and {\rm($\mathbf{H}_{2}$)},
\begin{eqnarray*}
\|X^{\epsilon}(t)-Z^{\epsilon}(t)\|^{2}&\leq& (1+r)\Big[\|\Lambda(t)\|^{2}+\frac{\|\int_{0}^{t}C(t-s)B[X^{\epsilon}(s)-Z^{\epsilon}(s)]ds\|^{2}}{r}\Big]\\
&\leq&(1+r)\Big[\|\Lambda(t)\|^{2}+\frac{M_{1}M_{B}^{2}T\int_{0}^{t}\|X^{\epsilon}(s)-Z^{\epsilon}(s)\|^{2}ds}{r}\Big],
\end{eqnarray*}
where
\begin{eqnarray*}
\Lambda(t)&=&\epsilon\int_{0}^{t}S(t-s)[F(s,X^{\epsilon}(s),\mathcal{L}(X^{\epsilon}(s)))-\bar{F}(Z^{\epsilon}(s),\mathcal{L}(Z^{\epsilon}(s)))]ds\\
&+&\sqrt{\epsilon}\int_{0}^{t}S(t-s)[G(s,X^{\epsilon}(s),\mathcal{L}(X^{\epsilon}(s)))-\bar{G}(Z^{\epsilon}(s),\mathcal{L}(Z^{\epsilon}(s)))]dW(s)\\
&+&\sqrt{\epsilon}\int_{0}^{t}\int_{Z}S(t-s)[J(s-,X^{\epsilon}(s-),\mathcal{L}(X^{\epsilon}(s-)),z)-\bar{J}(Z^{\epsilon}(s-),\mathcal{L}(Z^{\epsilon}(s-)),z)]\tilde{N}(ds,dz).
\end{eqnarray*}
Taking $r=\frac{\sqrt{M_{1}}M_{B}T}{1-\sqrt{M_{1}}M_{B}T}$, then
\begin{eqnarray*}
&&\mathbb{E}\sup_{0\leq s\leq t}\|X^{\epsilon}(s)-Z^{\epsilon}(s)\|^{2}\\
&\leq&(1+r)\Big[\mathbb{E}\sup_{0\leq s\leq t}\|\Lambda(s)\|^{2}+\frac{M_{1}M_{B}^{2}T^{2}\mathbb{E}\sup_{0\leq s\leq t}\|X^{\epsilon}(s)-Z^{\epsilon}(s)\|^{2}}{r}\Big]\\
&\leq&\frac{1}{1-\sqrt{M_{1}}M_{B}T}\mathbb{E}\sup_{0\leq s\leq t}\|\Lambda(s)\|^{2}+\sqrt{M_{1}}M_{B}T\mathbb{E}\sup_{0\leq s\leq t}\|X^{\epsilon}(s)-Z^{\epsilon}(s)\|^{2}.
\end{eqnarray*}
Further,
\begin{eqnarray}\label{AverXZ}
\mathbb{E}\sup_{0\leq s\leq t}\|X^{\epsilon}(s)-Z^{\epsilon}(s)\|^{2}\leq\frac{1}{(1-\sqrt{M_{1}}M_{B}T)^{2}}\mathbb{E}\sup_{0\leq s\leq t}\|\Lambda(s)\|^{2}.
\end{eqnarray}
By basic inequality,
\begin{eqnarray}\label{Lam}
&&\mathbb{E}\sup_{0\leq s\leq t}\|\Lambda(s)\|^{2}\nonumber\\
&\leq& 3\epsilon^{2}\mathbb{E}\sup_{0\leq s\leq t}\Big\|\int_{0}^{s}S(s-u)[F(u,X^{\epsilon}(u),\mathcal{L}(X^{\epsilon}(u)))-\bar{F}(Z^{\epsilon}(u),\mathcal{L}(Z^{\epsilon}(u-)))]du\Big\|^{2}\nonumber\\
&&+3\epsilon\mathbb{E}\sup_{0\leq s\leq t}\Big\|\int_{0}^{s}S(s-u)[G(u,X^{\epsilon}(u),\mathcal{L}(X^{\epsilon}(u)))-\bar{G}(Z^{\epsilon}(s),\mathcal{L}(Z^{\epsilon}(s)))]dW(u)\Big\|^{2}\nonumber\\
&&+3\epsilon\mathbb{E}\sup_{0\leq s\leq t}\Big\|\int_{0}^{s}\int_{Z}S(s-u)[J(u-,X^{\epsilon}(u-),\mathcal{L}(X^{\epsilon}(u-)),z)\nonumber\\
&&~~-\bar{J}(Z^{\epsilon}(u-),\mathcal{L}(Z^{\epsilon}(u-)),z)]\tilde{N}(du,dz)\Big\|^{2}\triangleq \sum_{i=1}^{3}\Pi_{i}.
\end{eqnarray}
For $\Pi_{1}$, by ($\mathbf{H}_{1}$) and ($\mathbf{H}_{3}$) and ($\mathbf{H}_{4}$)(i), we have
\begin{eqnarray}\label{Aver5}
\Pi_{1}&\leq&3\epsilon^{2}\mathbb{E}\sup_{0\leq s\leq t}\Big\|\int_{0}^{s}S(s-u)[F(u,X^{\epsilon}(u),\mathcal{L}(X^{\epsilon}(u)))-F(u,Z^{\epsilon}(u),\mathcal{L}(Z^{\epsilon}(u)))]du\Big\|^{2}\nonumber\\
&&+3\epsilon^{2}\mathbb{E}\sup_{0\leq s\leq t}\Big\|\int_{0}^{s}S(s-u)[F(u,Z^{\epsilon}(u),\mathcal{L}(Z^{\epsilon}(u)))-\bar{F}(Z^{\epsilon}(u),\mathcal{L}(Z^{\epsilon}(u)))]du\Big\|^{2}\nonumber\\
&\leq&3\epsilon^{2}M_{1}t\mathbb{E}\int_{0}^{t}\Psi(\|X^{\epsilon}(u)-Z^{\epsilon}(u)\|^{2}+\mathcal{W}_{2}(\mathcal{L}(X^{\epsilon}(u)),\mathcal{L}(Z^{\epsilon}(u)))^{2})du\nonumber\\
&&+3\epsilon^{2}M_{1}t^{2}\mathbb{E}\frac{1}{t}\int_{0}^{t}\|F(u,Z^{\epsilon}(u),\mathcal{L}(Z^{\epsilon}(u)))-\bar{F}(Z^{\epsilon}(u),\mathcal{L}(Z^{\epsilon}(u)))\|^{2}du\nonumber\\
&\leq&3\epsilon^{2}M_{1}t\int_{0}^{t}\Psi(2\mathbb{E}\sup_{0\leq v\leq u}\|X^{\epsilon}(v)-Z^{\epsilon}(v)\|^{2})du\nonumber\\
&&+3\epsilon^{2}M_{1}t^{2}\phi_{1}(T)\psi(2\mathbb{E}\sup_{0\leq s\leq t}\|Z^{\epsilon}(s)\|^{2}).
\end{eqnarray}
For $\Pi_{2}$ and $\Pi_{4}$, by a similar proof as (\ref{ER16}) and (\ref{ER17}), together with ($\mathbf{H}_{4}$)(ii), (iii), we have
\begin{eqnarray}\label{Aver6}
\Pi_{2}&\leq&3\epsilon\mathbb{E}\sup_{0\leq s\leq t}\Big\|\int_{0}^{s}S(s-u)[G(u,X^{\epsilon}(u),\mathcal{L}(X^{\epsilon}(u)))-G(u,Z^{\epsilon}(u),\mathcal{L}(Z^{\epsilon}(u)))]dW(u)\Big\|^{2}\nonumber\\
&&+3\epsilon\mathbb{E}\sup_{0\leq s\leq t}\Big\|\int_{0}^{s}S(s-u)[G(u,Z^{\epsilon}(u),\mathcal{L}(Z^{\epsilon}(u)))-\bar{G}(Z^{\epsilon}(u),\mathcal{L}(Z^{\epsilon}(u)))]dW(u)\Big\|^{2}\nonumber\\
&\leq&3\epsilon M_{1}\int_{0}^{t}\mathbb{E}\|G(u,X^{\epsilon}(u),\mathcal{L}(X^{\epsilon}(u)))-G(u,Z^{\epsilon}(u),\mathcal{L}(Z^{\epsilon}(u)))\|_{L_{2}^{0}}^{2}du\nonumber\\
&&+3\epsilon M_{1}\mathbb{E}\int_{0}^{t}\|G(u,Z^{\epsilon}(u),\mathcal{L}(Z^{\epsilon}(u)))-\bar{G}(Z^{\epsilon}(u),\mathcal{L}(Z^{\epsilon}(u)))\|_{L_{2}^{0}}^{2}du\nonumber\\
&\leq&3\epsilon M_{1}\mathbb{E}\int_{0}^{t}\Psi(\|X^{\epsilon}(u)-Z^{\epsilon}(u)\|^{2}+\mathcal{W}_{2}(\mathcal{L}(X^{\epsilon}(u)),\mathcal{L}(Z^{\epsilon}(u)))^{2})du\nonumber\\
&&+3\epsilon M_{1}t\mathbb{E}\frac{1}{t}\int_{0}^{t}\|G(u,Z^{\epsilon}(u),\mathcal{L}(Z^{\epsilon}(u)))-\bar{G}(Z^{\epsilon}(u),\mathcal{L}(Z^{\epsilon}(u)))\|_{L_{2}^{0}}^{2}du\nonumber\\
&\leq&3\epsilon M_{1}\int_{0}^{t}\Psi(2\mathbb{E}\sup_{0\leq v\leq u}\|X^{\epsilon}(v)-Z^{\epsilon}(v)\|^{2})du\nonumber\\
&&+3\epsilon M_{1}t\phi_{2}(T)\psi(2\mathbb{E}\sup_{0\leq s\leq t}\|Z^{\epsilon}(s)\|^{2}),
\end{eqnarray}
and
\begin{eqnarray}\label{Aver7}
\Pi_{3}&\leq&3\epsilon\mathbb{E}\sup_{0\leq s\leq t}\Big\|\int_{0}^{s}\int_{Z}S(s-u)[J(u-,X^{\epsilon}(u-),\mathcal{L}(X^{\epsilon}(u-)),z)\nonumber\\
&&~~-J(u-,Z^{\epsilon}(u-),\mathcal{L}(Z^{\epsilon}(u-)),z)]\tilde{N}(du,dz)\Big\|^{2}\nonumber\\
&&+3\epsilon\mathbb{E}\sup_{0\leq s\leq t}\Big\|\int_{0}^{s}\int_{Z}S(s-u)[J(u-,Z^{\epsilon}(u-),\mathcal{L}(Z^{\epsilon}(u-)),z)\nonumber\\
&&~~-\bar{J}(Z^{\epsilon}(u-),\mathcal{L}(Z^{\epsilon}(u-)),z)]\tilde{N}(du,dz)\Big\|^{2}\nonumber\\
&\leq&3\epsilon M_{1}\mathbb{E}\int_{0}^{t}\int_{Z}\|J(u-,X^{\epsilon}(u-),\mathcal{L}(X^{\epsilon}(u-)),z)\nonumber\\
&&-J(u-,Z^{\epsilon}(u-),\mathcal{L}(Z^{\epsilon}(u-)),z)\|^{2}\nu(dz)du+3\epsilon M_{1}\mathbb{E}\int_{0}^{t}\int_{Z}\|J(u-,Z^{\epsilon}(u-),\mathcal{L}(Z^{\epsilon}(u-)),z)\nonumber\\
&&-\bar{J}(Z^{\epsilon}(u-),\mathcal{L}(Z^{\epsilon}(u-)),z)\|^{2}\nu(dz)du\nonumber\\
&\leq&3\epsilon M_{1}\int_{0}^{t}\Psi(\|X^{\epsilon}(u)-Z^{\epsilon}(u)\|^{2}+\mathcal{W}_{2}(\mathcal{L}(X^{\epsilon}(u)),\mathcal{L}(Z^{\epsilon}(u)))^{2})du\nonumber\\
&&+3\epsilon M_{1}t\mathbb{E}\frac{1}{t}\int_{0}^{t}\int_{Z}\|J(u-,Z^{\epsilon}(u-),\mathcal{L}(Z^{\epsilon}(u-)),z)-\bar{J}(Z^{\epsilon}(u-),\mathcal{L}(Z^{\epsilon}(u-)),z)\|^{2}\nu(dz)du\nonumber\\
&\leq&3\epsilon M_{1}\int_{0}^{t}\Psi(2\mathbb{E}\sup_{0\leq v\leq u}\|X^{\epsilon}(v)-Z^{\epsilon}(v)\|^{2})du+3\epsilon M_{1}t\phi_{3}(T)\psi(2\mathbb{E}\sup_{0\leq s\leq t}\|Z^{\epsilon}(s)\|^{2}).
\end{eqnarray}
Substituting (\ref{Lam})-(\ref{Aver7}) into (\ref{AverXZ}), and together with (\ref{ES}), we have

\begin{eqnarray*}
&&\mathbb{E}\sup_{0\leq s\leq t}\|X^{\epsilon}(s)-Z^{\epsilon}(s)\|^{2}\\
&\leq&\frac{3\epsilon^{2}M_{1}t}{(1-\sqrt{M_{1}}M_{B}T)^{2}}\int_{0}^{t}\Psi(2\mathbb{E}\sup_{0\leq v\leq u}\|X^{\epsilon}(v)-Z^{\epsilon}(v)\|^{2})du\\
&&+\frac{3\epsilon^{2}M_{1}t^{2}}{(1-\sqrt{M_{1}}M_{B}T)^{2}}\phi_{1}(T)\psi(2\mathbb{E}\sup_{0\leq s\leq t}\|Z^{\epsilon}(s)\|^{2})\\
&&+\frac{3\epsilon M_{1}}{(1-\sqrt{M_{1}}M_{B}T)^{2}}\int_{0}^{t}\Psi(2\mathbb{E}\sup_{0\leq v\leq u}\|X^{\epsilon}(v)-Z^{\epsilon}(v)\|^{2})du\\
&&+\frac{3\epsilon M_{1}t}{(1-\sqrt{M_{1}}M_{B}T)^{2}}\phi_{2}(T)\psi(2\mathbb{E}\sup_{0\leq s\leq t}\|Z^{\epsilon}(s)\|^{2})\\
&&+\frac{3\epsilon M_{1}}{(1-\sqrt{M_{1}}M_{B}T)^{2}}\int_{0}^{t}\Psi(2\mathbb{E}\sup_{0\leq v\leq u}\|X^{\epsilon}(v)-Z^{\epsilon}(v)\|^{2})du\\
&&+\frac{3\epsilon M_{1}t}{(1-\sqrt{M_{1}}M_{B}T)^{2}}\phi_{3}(T)\psi(2\mathbb{E}\sup_{0\leq s\leq t}\|Z^{\epsilon}(s)\|^{2})\\
&\leq&\frac{3\epsilon^{2}M_{1}t}{(1-\sqrt{M_{1}}M_{B}T)^{2}}\int_{0}^{t}[\beta+\beta(2\mathbb{E}\sup_{0\leq v\leq u}\|X^{\epsilon}(v)-Z^{\epsilon}(v)\|^{2})]du\\
&&+3\epsilon^{2}M_{1}t^{2}\phi_{1}(T)\psi(2\mathbb{E}\sup_{0\leq s\leq t}\|Z^{\epsilon}(s)\|^{2})\\
&&+\frac{3\epsilon M_{1}}{(1-\sqrt{M_{1}}M_{B}T)^{2}}\int_{0}^{t}[\beta+\beta(2\mathbb{E}\sup_{0\leq v\leq u}\|X^{\epsilon}(v)-Z^{\epsilon}(v)\|^{2})]du\\
&&+\frac{3\epsilon M_{1}t}{(1-\sqrt{M_{1}}M_{B}T)^{2}}\phi_{2}(T)\psi(2\mathbb{E}\sup_{0\leq s\leq t}\|Z^{\epsilon}(s)\|^{2})\\
&&+\frac{3\epsilon M_{1}}{(1-\sqrt{M_{1}}M_{B}T)^{2}}\int_{0}^{t}[\beta+\beta(2\mathbb{E}\sup_{0\leq v\leq u}\|X^{\epsilon}(v)-Z^{\epsilon}(v)\|^{2})]du\\
&&+\frac{3\epsilon M_{1}t}{(1-\sqrt{M_{1}}M_{B}T)^{2}}\phi_{3}(T)\psi(2\mathbb{E}\sup_{0\leq s\leq t}\|Z^{\epsilon}(s)\|^{2})\\
&\leq&C_{1}\epsilon t+C_{3}\int_{0}^{t}\mathbb{E}\sup_{0\leq v\leq u}\|X^{\epsilon}(v)-Z^{\epsilon}(v)\|^{2}du,
\end{eqnarray*}
where
\begin{eqnarray*}
C_{1}&=&3\epsilon M_{1}t\beta+3\epsilon M_{1}t\phi_{1}(T)\psi(2\mathbb{E}\sup_{0\leq s\leq t}\|Z^{\epsilon}(s)\|^{2})\\
&&+3\epsilon M_{1}\beta+3 M_{1}\phi_{2}(T)\psi(2\mathbb{E}\sup_{0\leq s\leq t}\|Z^{\epsilon}(s)\|^{2})\\
&&+3\epsilon M_{1}\beta+3 M_{1}\phi_{3}(T)\psi(2\mathbb{E}\sup_{0\leq s\leq t}\|Z^{\epsilon}(s)\|^{2})\\
C_{3}&=&6\epsilon^{2}M_{1}ta+6M_{1}\epsilon \beta+6M_{1}C\epsilon \beta.
\end{eqnarray*}
Gronwall's inequality yields
\begin{eqnarray*}
\mathbb{E}\sup_{0\leq s\leq t}\|X^{\epsilon}(s)-Z^{\epsilon}(s)\|^{2}&\leq&C_{1}\epsilon te^{C_{3}t}.
\end{eqnarray*}
Choosing $\alpha\in(0,1)$ and $L>0$ such that for any $t\in[0,L\epsilon^{-\alpha}]\subset[0,T]$, we have
\begin{eqnarray*}
\mathbb{E}\sup_{t\in[0,L\epsilon^{-\alpha}]}\|X^{\epsilon}(s)-Z^{\epsilon}(s)\|^{2}&\leq&C_{4}L\epsilon^{1-\alpha},
\end{eqnarray*}
where $C_{4}=C_{1}e^{C_{3}L\epsilon^{-\alpha}}$.
Hence, for any $\delta_{1}>0$, choosing $\epsilon_{1}\in(0,\epsilon_{0}]$ such that for any $\epsilon\in(0,\epsilon_{1}]$ and $t\in[0,L\epsilon^{-\alpha}]$, we have
\begin{eqnarray*}
\mathbb{E}\sup_{t\in[0,L\epsilon^{-\alpha}]}\|X^{\epsilon}(s)-Z^{\epsilon}(s)\|^{2}&\leq&\delta_{1}.
\end{eqnarray*}
\end{proof}

\end{document}